\documentclass[reqno]{amsart}
\usepackage{a4wide}
\usepackage{amssymb, amsmath, color}
\usepackage{mathrsfs}
\usepackage{graphicx}

\theoremstyle{plain}
\newtheorem{theorem}{Theorem}

\newtheorem{lemma}[theorem]{Lemma}
\newtheorem{proposition}[theorem]{Proposition}

\theoremstyle{remark}

\def\Z{{\mathbb Z}}

\def\R{{\mathbb R}}

\newcommand{\T}{\mathbb{T}}

\DeclareMathOperator{\sinc}{sinc}

\newcommand{\ddt}{\left(\!\frac{d}{dt}\!\right)}
\newcommand{\ddx}{\left(\!\frac{d}{dx}\!\right)}

\usepackage[colorlinks,linkcolor={blue},
citecolor={blue},urlcolor={red},]{hyperref}

\begin{document}

 \title[On the constant in a transference inequality]{On the constant in a transference inequality for the vector-valued Fourier transform}
 \author{Dion Gijswijt \& Jan van Neerven }
 \address{Delft University of Technology \\ Faculty EEMCS/DIAM \\ P.O. Box 5031 \\ 2600 GA Delft \\ The Netherlands}
 \email{D.C.Gijswijt/J.M.A.M.vanNeerven@TUDelft.nl}
 \date{\today}

\begin{abstract} The standard proof of the equivalence of Fourier type on $\R^d$ and on the torus $\T^d$ is usually stated 
in terms of an implicit constant, defined as the minimum of a sum of powers of sinc functions. In this note
we compute this minimum explicitly.  
\end{abstract}

\maketitle

\section{Introduction}

The motivation of this paper comes from a well-known transference result for the vector-valued Fourier transform.
Let $X$ be a complex Banach space. The {\em Fourier transform} of a function $f\in L^1(\R^d;X)$ is defined
by $$\mathscr{F}_{\R^d}f(\xi) := \int_{\R^d} e^{-2\pi ix\cdot \xi}f(x)\,dx, \quad \xi\in \R^d.$$
Likewise, the {\em Fourier transform} of a function $f\in { L^1(\T^d;X)}$ is defined by
$$ \mathscr{F}_{ \T^d}f({ k}) := { \int_{\T^d}} e^{-2\pi i k\cdot t} f{(t)\,dt}, \quad { k\in \Z^d}.$$
 
\begin{proposition}\label{prop:1}
Let $X$ be a complex Banach space, fix $d\ge 1$ and $p\in (1, 2]$, and let $\frac1p+\frac1q = 1$.
The following assertions are equivalent:
\begin{enumerate}
\item[(i)] 
$\mathscr{F}_{\R^d}$ extends to a bounded operator from $L^p(\R^d;X)$ into $L^q(\R^d;X)$;
\item[(ii)]  
$\mathscr{F}_{\T^d}$ extends to a bounded operator from $L^p(\T^d;X)$ into $\ell^q(\Z^d;X)$.
\end{enumerate}
In this situation, denoting the norms of these extensions by $\varphi_{p,X}(\R^d)$ and $\varphi_{p,X}(\T^d)$, we have
\begin{align*}
   \varphi_{p,X}(\R^d)  \leq \varphi_{p,X}(\T^d) \leq C_{q}^{-d/q} \varphi_{p,X}(\R^d), \end{align*}
where $C_{q}$ is the global minimum of the periodic function
$$ x\mapsto \sum_{m \in \Z} \Big|\frac{\sin(\pi(x+m))}{\pi(x+m)}\Big|^{q}, \quad x\in\R.$$ 
\end{proposition}
This function, as well as several others considered below, have removable singularities. It is understood that we 
will always be working with their unique continuous extensions.

A complex Banach space $X$ which has the equivalent properties (i) and (ii) is said to have {\em Fourier type $p$}; 
this notion has been introduced in 
\cite{Pee}.
Proposition~\ref{prop:1} goes back to \cite{Kwa}; in its stated form the result can be found in \cite{Gar, Kon}. Related results
may be found in \cite{Bou}. These
references do not comment on the location of the global minimum. A quick computer plot (see Figure \ref{fig1}) 
suggests that the minimum is taken in the points
$\frac12 + \Z$. To actually {\em prove} this turns out to be surprisingly difficult. This is the modest objective of the present note:

\begin{proposition}\label{prop} For every real number $r\ge 1$, 
the function $f_r:[0,1]\to \R$ defined by  
$$f_r(x):= \sum_{m \in \Z} {\Big|}\frac{\sin(\pi(x+m))}{\pi(x+m)}{\Big|}^{2r}, \quad x\in [0,1],$$
has a global minimum at $x=\frac12$.
\end{proposition}

Our proof has developed essentially by trial and error. We believe it is perfectly possible that a
truly pedestrian proof can be given, but we failed to find one despite many hours of efforts.

As a consequence of Proposition~\ref{prop}
 we obtain the explicit estimate
$$  \varphi_{p,X}(\R^d)  \leq \varphi_{p,X}(\T^d) 
\leq 
 \frac{\pi^d}{ \bigl(2(2^q-1)\zeta(q)\bigr)^{d/q}}\cdot\varphi_{p,X}(\R^d), $$
noting that 
$$ \sum_{m \in \Z}\frac1{|\frac12+m|^{q}} 
= 2(2^q-1)\zeta(q).$$
For even integers $q = 2n$, the constant on the right-hand side may 
be evaluated explicitly in terms of the Bernoulli numbers. To further estimate this constant,
recall that for any $x\in \ell^2(\Z)$ the function 
$q\mapsto ||x||_q:=(\sum_{m\in \Z}|x_m|^q)^{1/q}$ is decreasing on $[2,\infty)$ 
and $\lim_{q\to\infty}||x||_q=\sup_{i\in \Z}|x_i|$. Taking $x_m:=|\tfrac{1}{2}+m|^{-1}$
we find $(\sum_{m\in\Z}|\tfrac{1}{2}+m|^{-q})^{1/q}\geq 2$ for every $q\geq 2$, and hence  
in particular
$$ \varphi_{p,X}(\R^d)  \leq \varphi_{p,X}(\T^d) 
\leq (\tfrac{1}{2}\pi)^d \varphi_{p,X}(\R^d).$$

\begin{figure}\label{fig1}
\begin{center}
\includegraphics[width=7cm]{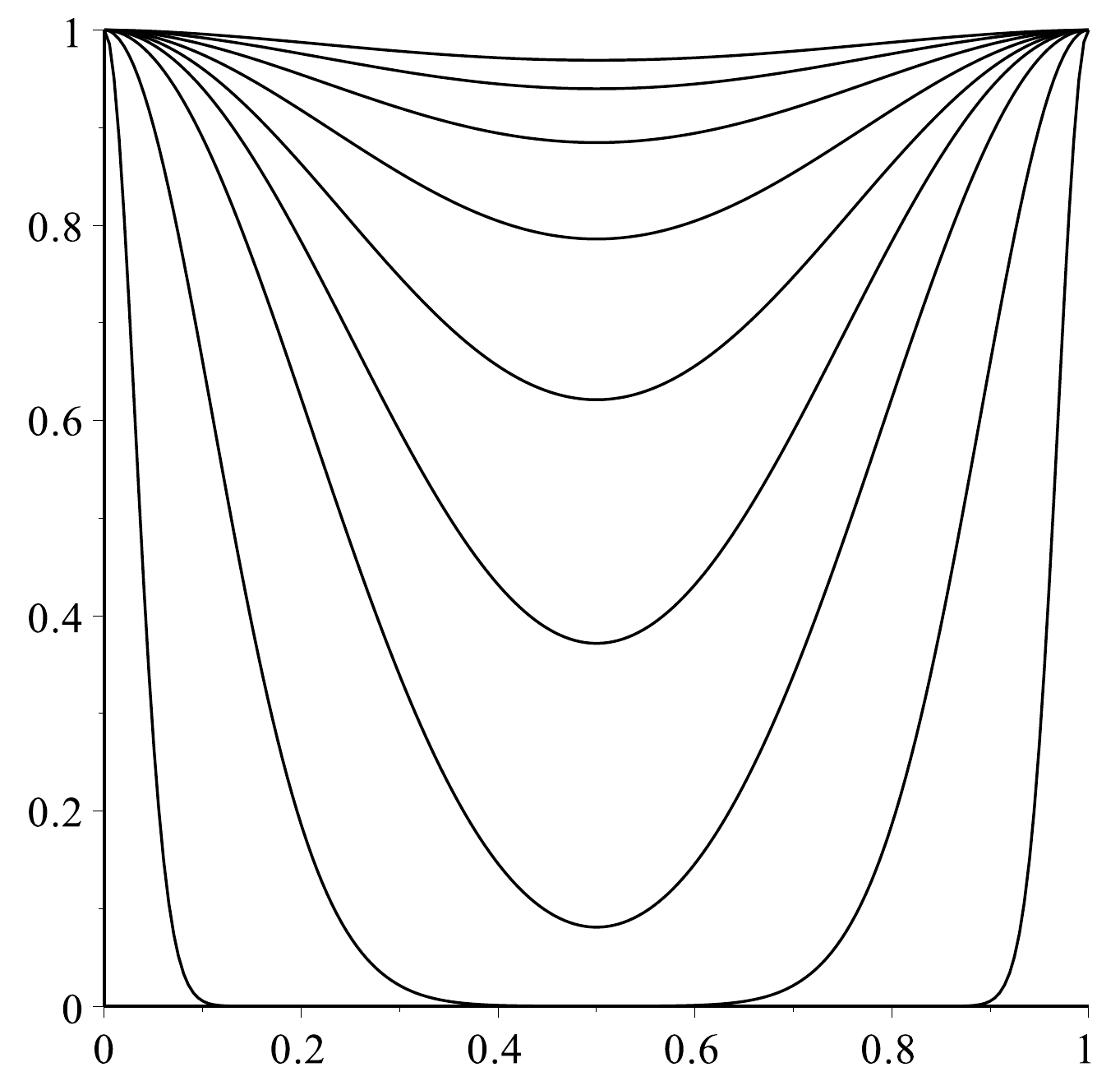}
\end{center}
\caption{A plot of $f_r$, where $r=1.02^{k}$ for $k=1,2,4,\dots,256$.}
\end{figure} 
 
\section{The main result}

The proof of the proposition is based on the following lemmas. The main idea is contained in the first lemma.

\begin{lemma}\label{lem:3}
Let $g:\R_+\to\R_+$ be a { non-de}creasing convex function, and let 
$x_1, \dots ,x_n\in\R_+$ and $ y_1,\dots,y_n \in \R_+$ be such that
\begin{enumerate}
 \item[(i)] $x_1+ \cdots +x_n \ge  y_1+\cdots+y_n$;
 \item[(ii)] there exists $t\in\R_+$ such that 
\begin{itemize}
 \item 
$x_i\le y_i$ if $ y_i<t$;
 \item
$x_i\ge y_i$ if $ y_i\ge t$.
\end{itemize}
\end{enumerate}
Then $g(x_1)+\cdots +g(x_n)\ge g(y_1)+\cdots+g(y_n)$.
\end{lemma}
\begin{proof}
We will prove the lemma by induction on $n$. The case $n=1$ is clear: $x_1\geq y_1$ implies that 
$g(x_1)\geq g(y_1)$ since $g$ is non-decreasing. Suppose now that the lemma has been proved for $n=1,\dots,m-1$. 

If $x_i=y_i$ for some index $1\le i\le m$, then we may remove $x_i$ and $y_i$ and apply the induction hypothesis. 

If $x_i\geq y_i$ for every index $1\le i\le m$, 
then again the result is immediate since $g$ is non-decreasing. Therefore, we may assume that $x_i<y_i$ for 
some index $1\le i\le m$. Then, by the first condition in the lemma, there is also an index $j$ for which $x_j>y_j$. 
By the second condition in the lemma we then have $x_i<y_i<t\leq y_j<x_j$. 

Let $\epsilon:=\min(y_i-x_i,x_j-y_j)$ and define $x'_i:=x_i+\epsilon$, $x'_j:=x_j-\epsilon$, and $x'_k:=x_k$ for all 
other indices. Then $x'_1,\dots,x'_m,y_1,\dots,y_m$ satisfy the conditions in the lemma (with the same $t$) and 
$x'_i=y_i$ or $x'_j=y_j$. Hence, by the induction hypothesis, we have 
\begin{equation}\label{first}
g(x'_1)+\cdots +g(x'_m)\ge g(y_1)+\cdots+g(y_m).
\end{equation}
Since $x_i\leq x'_i\leq x'_j\leq x_j$, we can write $x'_i=\lambda x_i+(1-\lambda)x_j$ for some $\lambda\in [0,1]$. 
Since $x'_j=x_i+x_j-x'_i$, we have $x'_j=(1-\lambda) x_i+\lambda x_j$. By the convexity of $g$ it follows that
\begin{equation}\label{second}
g(x'_i)+g(x'_j)\leq \left(\lambda g(x_i)+(1-\lambda) g(x_j)\right)+ \left((1-\lambda) g(x_i)+\lambda g(x_j)\right)=g(x_i)+g(x_j).
\end{equation}
Combining inequalities (\ref{first}) and (\ref{second}) we obtain the lemma for $n=m$, thus completing the induction step.
 \end{proof}

In order to apply this lemma we need a number of technical facts.  
The first  (cf. \cite[(6.14)]{Gar})  is elementary and is left as an exercise. 

\begin{lemma}\label{lem:1} 
$f_1(x) = 1$ for all $x\in [0,1]$.
\end{lemma}

Let $h:\R\to\R$ be defined by $$h(x):= \sinc^2(\pi x) =  \Big(\frac{\sin(\pi x)}{\pi x}\Big)^2, \quad x\in\R.$$

\begin{lemma}\label{lem:2} Let $r\ge 1$.
 The following assertions hold on the interval $[0,1]$: 
 \begin{enumerate}
  \item[\rm(i)] the function $h(x)+h(x-1)$ has a global minimum at $x=\frac12$;
  \item[\rm(ii)] for all $m=1,2,3,\dots$, $h(x+m)+h(x-(m+1))$ has a global maximum at $x=\frac12$;
  \item[\rm(iii)] the function $$h(x) + h(x-1) - (h(x)^r+h(x-1)^r)^{1/r} $$ has a global maximum at $x=\frac12$;
  \item[\rm(iv)] for all $m=1,2,3,\dots$ and $r\ge 1$,  $$(h(x+m)+h(x-(m+1)))^r - h(x+m)^r-h(x-(m+1))^r$$ has a global maximum at $x=\frac12$.
 \end{enumerate}
\end{lemma}

Assuming the lemmas for the moment, let us first show how the proposition can be deduced from them.

\begin{proof}[Proof of Proposition~\ref{prop}] Fix $r\ge 1$ and 
set, for $x\in [0,1]$, 
$$s_m(x):=h(x+m)+h(x-(m+1)) \qquad (m=0,1,2,\dots)$$
and 
$$\widetilde s_0(x) :=  ((h(x))^r+(h(x-1))^r)^{1/r}.$$ 
In view of part (iv) of Lemma~\ref{lem:2}
it suffices to prove that 
$${\widetilde s_0}^{\,r}+s_1^r+s_2^r+\cdots $$ 
has a global minimum at $x=\frac12$. 

Fix an arbitrary $x \in [0,1]$ and set
$$x_m:=s_m(x), \quad 
y_m:=s_m(\tfrac12) \qquad (m=0,1,2,\dots)$$
and 
$$ \widetilde x_0 := ((h(x))^r+h(x-1)^r)^{1/r},\quad \widetilde y_0 := ((h(\tfrac12))^r+h(-\tfrac12)^r)^{1/r}.$$
In view of parts (i) and (ii) of Lemma~\ref{lem:2} we have 
\begin{align}\label{eq:1}
x_0\geq y_0, \quad  x_i\leq y_i \qquad(i=1,2,\dots)
\end{align}
Lemma~\ref{lem:1} implies 
\begin{align}\label{eq:2} 
x_0+x_1+x_2+\cdots=y_0 +y_1+y_2+\cdots 
\end{align}
By \eqref{eq:1} and \eqref{eq:2},
\begin{align}\label{eq:2a} 
x_0+x_1+\cdots+x_n \ge y_0 +y_1+\cdots+y_n \qquad (n=0,1,2,\ldots)
\end{align}
Part (iii) of Lemma~\ref{lem:2} implies
\begin{align}\label{eq:3} 
\widetilde x_0 -x_0 \ge \widetilde y_0-y_0.
\end{align}
{ By \eqref{eq:2a} and \eqref{eq:3},}
\begin{equation}\label{eq:5}  \widetilde x_0+x_1+\cdots+x_n\ge \widetilde y_0+y_1+\cdots+y_n \qquad (n=0,1,2,\ldots)
\end{equation}
Finally, by \eqref{eq:1} and \eqref{eq:3},
\begin{align}\label{eq:4} 
\widetilde x_0  \ge \widetilde y_0.
\end{align}

A simple calculation shows that $\widetilde y_0>\tfrac{4}{\pi^2}$ and $y_i<\tfrac{4}{\pi^2}$ for $i=1,2,\dots$. 
Taking $t=\tfrac{4}{\pi^2}$ in Lemma~\ref{lem:3} and $g(x):=x^r$ now implies, by 
virtue of \eqref{eq:1}, \eqref{eq:5}, and \eqref{eq:4},  that 
$$\widetilde x_0^{\,r}+x_1^r+\cdots+x_n^r\geq \widetilde y_0^{\,r}+y_1^r+\cdots+y_n^r$$
holds for every $n$. Taking limits for $n\to\infty$ completes the proof.
\end{proof}

\section{Proof of Lemma~\ref{lem:2}}

This section is devoted to the proof of Lemma~\ref{lem:2}, which is based on the following observations: 

\begin{lemma}\label{lem:6} On the interval $[0,1]$:
\begin{itemize}
  \item[(i)] $\displaystyle\frac{\cos(\frac12\pi x)}{1-x^2}$ takes a global maximum at $x=0$;\smallskip
  \item[(ii)] $\displaystyle\frac{(x^2+1)\cos^2(\frac12\pi x)}{(1-x^2)^2}$ takes a global minimum at $x=0$.
\end{itemize}
\end{lemma}
\begin{proof}
We start by showing that 
\begin{equation}\label{basicinequality}
\sqrt{2}\sin(\tfrac{1}{4}\pi x)\geq x\quad\text{for all $x\in [0,1]$}.
\end{equation}
To this end, consider the function $f(x):=\sqrt{2}\sin(\tfrac{1}{4}\pi x)-x$. Observe that 
$f'(x)=\tfrac{\pi\sqrt{2}}{4}\cos(\tfrac{1}{4}\pi x)-1$ is decreasing on $[0,1]$, hence $f$ is 
concave. Since $f(0)=f(1)=0$ this implies that $f(x)\geq 0$ for $x\in [0,1]$, which proves the claim.

\smallskip
(i): \  
The value at $x=0$ of the given function equals $1$, so it suffices to show that 
$\cos(\tfrac{1}{2}\pi x) \leq  1-x^2$ for all $ x\in [0,1]$. This follows from the double-angle 
formula for cosine and \eqref{basicinequality}: 
$$\cos(\tfrac{1}{2}\pi x)=1-2\sin^2(\tfrac{1}{4}\pi x)\leq 1-x^2.$$

\smallskip
 (ii): \  The given function has value $1$ at $x=0$, hence it suffices to show that for all $x\in [0,1]$, 
$${(x^2+1)\cos^2(\tfrac12\pi x)} \ge {(1-x^2)^2}.$$

On the interval $[\frac12,1]$ we substitute $x = 1-y$. We then must prove that for $y\in [0,\frac12]$,
$$(2-2y+y^2)\sin^2(\tfrac{1}{2}\pi y) \ge (2y-y^2)^2 .$$

Since $2y\in [0,1]$, we can use \eqref{basicinequality} to obtain $\sqrt{2}\sin(\tfrac{1}{4}\pi \cdot2y)\geq 2y$, 
and hence $\sin^2(\tfrac{1}{2}\pi y)\geq 2y^2$. This implies that 
$$
(2-2y+y^2)\sin^2(\tfrac{1}{2}\pi y)\geq (2-2y+y^2)(2y^2)=(y^2+(2-y)^2)y^2\geq (2-y)^2y^2=(2y-y^2)^2,
$$
which concludes the proof on the interval $[\frac12,1]$. 

For $x\in [0,\frac12]$ we have 
\begin{align*}
 (x^2+1)\cos^2(\tfrac12\pi x) & \ge (x^2+1)\Big(1-\frac{\pi^2}{8} x^2\Big)^2
 \\ & = (x^2+1)(1-\frac{\pi^2}{4} x^2 + \frac{\pi^4}{64} x^4)
 \\ & \ge 1 +(1-\frac{\pi^2}{4}) x^2 + (\frac{\pi^4}{64} -\frac{\pi^2}{4}) x^4 
\\ & = 1 +(1-\frac{\pi^2}{4}) x^2 + (\frac{\pi^4}{64} -\frac{\pi^2}{4}-1) x^4  + x^4 
\\ & \ge 1 +\Big[(1-\frac{\pi^2}{4}) + \frac14(\frac{\pi^4}{64} -\frac{\pi^2}{4}-1)\Big] x^2  + x^4 
\\ & \ge 1-2x^2 + x^4 
 \\ & = (1-x^2)^2,
 \end{align*}
noting that $\frac{\pi^4}{64} -\frac{\pi^2}{4}-1 < 0$ and 
$(1-\frac{\pi^2}{4}) + \frac14(\frac{\pi^4}{64} -\frac{\pi^2}{4}-1)\approx -1.9537471\dots> -2$
\end{proof}

\begin{proof}[Proof of Lemma~\ref{lem:2}]
(i): \  We have
 $$ h(x)+h(x-1) = \frac{\sin^2(\pi x)}{\pi^2 x^2} + \frac{\sin^2(\pi x)}{\pi^2(x-1)^2} 
 = \frac{(2x^2-2x+1)\sin^2(\pi x)}{\pi^2 x^2(x-1)^2} =:g(x).$$
 We must show that $$f(x):= g(x+\frac12) = \frac{8}{\pi^2}\frac{4x^2 + 1}{(4x^2-1)^2 }\cos^2 (\pi x)$$
 has a global minimum in $x=0$ on the interval $[-\frac12,\frac12]$. 
 But this follows from Lemma~\ref{lem:6} and the fact that $f$ is even.

\smallskip
(ii):  \  For $m=1,2,3,\dots$ we have 
\begin{align*} h(x+m)+h(x-(m+1))
& = \frac{[2x^2-2x+(m+1)^2+m^2]\sin^2(\pi x)}{\pi^2 [(x+m)^2(x-(m+1))^2]} =:g_m(x). 
\end{align*}
We must show that $$f_m(x):= g_m(x+\frac12) 
= \frac{8}{\pi^2}\frac{4x^2  +  4m^2+4m + 1 }{[(4x^2 - (2m+1)^2]^2 }\cos^2 (\pi x)$$
 has a global maximum in $x=0$ on the interval $[-\frac12,\frac12]$. 
For this, it suffices to check that the functions
$$ \frac{4x^2  + 1 }{(4x^2 - M^2)^2 }\cos^2 (\pi x) \ \ \ \hbox{and} \ \ \
\frac{1 }{(4x^2 - M^2)^2 }\cos^2 (\pi x)$$
are decreasing on $[0,\frac12]$ for each $M\ge 3$, or equivalently, that 
$$ \frac{\sqrt{x^2  + 1 }}{M^2 - x^2 }\cos (\tfrac12\pi x) \ \ \ \hbox{and} \ \ \
\frac{1 }{M^2 - x^2 }\cos (\tfrac12\pi x)$$
are decreasing on $[0,1]$ for each $M\ge 3$. It suffices to prove this for the first function,
since this will immediately imply the result for the second function. 

Straightforward algebra shows that the derivative of the function
$$\psi_M(x) := \frac{\sqrt{x^2  + 1 }}{M^2 - x^2 }\cos (\tfrac12\pi x) $$ has a zero at $x$ if and only if
$$  2x( x^2+2+M^2)\cos (\tfrac12\pi x) = \pi(M^2- x^4+(M^2-1)x^2)\sin(\tfrac12\pi x).$$
But,
$$  2x(M^2+2+x^2)\cos (\tfrac12\pi x)
\le 2x(M^2+2+x^2)$$ and, since $0\le x\le 1$,
$$\pi(M^2- x^4+(M^2-1)x^2)x \le
\pi(M^2- x^4+(M^2-1)x^2)\sin(\tfrac12\pi x),$$
while also, using that $M\ge 3$ and $0\le x\le 1$, 
$$ 2(M^2+2+x^2) \le  2(M^2+2+(M^2-1) x^2) < \pi(M^2- 1+(M^2-1)x^2) 
\le \pi(M^2- x^4+(M^2-1)x^2) $$
since $2(M^2+2) < \pi(M^2-1)$ for $M\ge 3$.
It follows that the derivative of $\psi_M$
has no zeros on $(0,1]$, and then from $$\psi_M(0) = \frac{{ 1 }}{M^2} > 0 = \psi_M(1)$$
it follows that $\psi_M$ is decreasing on $[0,1]$.

\smallskip 
(iii): \ 
Proceeding as in (i),
we have
\begin{align*}
 &   h(x) + h(x-1) - ((h(x))^r+(h(x-1))^r)^{1/r}
 \\ & \quad = \frac1{\pi^2}\Big[\frac{1}{x^2} +\frac{1}{(1-x)^2} - \Big(\frac{1}{x^{2r}} +\frac{1}{(1-x)^{2r}}\Big)^{1/r}
\Big]\sin^{2}(\pi x) 
  =:g(x).
\end{align*}
 We must show that 
 \begin{align*}f(x):= g(\tfrac12+x) 
  = \frac1{\pi^2}\Big[ (\tfrac12-x)^2 + (\tfrac12+x)^2 - ((\tfrac12-x)^{2r} + (\tfrac12+x)^{2r})^{1/r} 
  \Big]\frac{\cos^{2}(\pi x)}{(\frac14-x^2)^2} 
 \end{align*}
 has a global maximum in $x=0$ on the interval $[-\frac12,\frac12]$. 
 The function $f$ is even, and by Lemma~\ref{lem:6}, ${\cos^2(\pi x)}/(\frac14-x^2)^2$ 
 takes its maximum at $x=0$.
It thus remains to show that on the interval $[0,\frac12]$ the function 
  $$
 \phi_r(x):= (\tfrac12-x)^2 + (\tfrac12+x)^2 - ((\tfrac12-x)^{2r} + (\tfrac12+x)^{2r})^{1/r}
 $$
  is decreasing on $[0,\frac12]$. The derivative of this function equals
 \begin{align*}
 \phi_r'(x) & =  4x - 2\big((\tfrac12-x)^{2r} + (\tfrac12+x)^{2r}\big)^{1/r-1}
 \big((\tfrac12+x)^{2r-1}-(\tfrac12-x)^{2r-1}\big).
 \end{align*}
To show that $\phi_r'(x)\le 0$ we must show that 
$$\big((\tfrac12+x)^{2r} + (\tfrac12-x)^{2r}\big)^{1/r-1}
 \big((\tfrac12+x)^{2r-1}-(\tfrac12-x)^{2r-1}\big) \ge 2x$$
for $x\in [0,\frac12]$, or, after substituting $a = \frac12+x$ and $b = \frac12-x$, that
$$ a^{2r-1} - b^{2r-1} \ge (a-b)\bigl( a^{2r} + b^{2r}\bigr)^{1-1/r}$$
for all $a\in [\frac12,1]$.
In view of 
\begin{align*}\bigl( a^{2r} + b^{2r}\bigr)^{1-1/r}
& = \Bigl[\bigl( a^{2r} + b^{2r}\bigr)^{1/(2r)}\Bigr]^{2r-2} \\ 
& \le  \Bigl[\bigl( a^{2r-1} + b^{2r-1}\bigr)^{1/(2r-1)}\Bigr]^{2r-2}
 =  \big(a^{2r-1} + b^{2r-1}\big)^{1- 1/(2r-1)},
\end{align*}
with  $p:=2r-1$ it suffices to show that 
 $$a^p - b^p  \ge (a-b)(a^p + b^p)^{1- 1/p}$$
 for all $a\ge b\ge 0$.
We can further simplify this upon dividing both sides by $b^p$. In the new variable $x= a/b$ we then have to prove that 
 $$x^p - 1  \ge (x-1)(x^p + 1)^{1- 1/p} $$
for all $x\ge 1$. 

Using that $(1+y)^\alpha \le 1+\alpha y$ for $y\ge 0$ and $0\le\alpha\le 1$, we have
 \begin{align*}
   (x-1)(x^p + 1)^{1- 1/p}
  = (x^p - x^{p-1})(1+x^{-p})^{1-1/p}
 \le (x^p - x^{p-1}) [ 1 + (1-\frac1p)\,x^{-p}].
\end{align*}
Therefore it remains to prove that for $x\ge 1$ and $p\ge 1$ we have 
$$ x^p - 1  \ge (x^p - x^{p-1}) [ 1 + (1-\frac1p)\,x^{-p}],$$
or, multiplying both sides with $x$, that
$$ x^{p+1} - x  \ge x^{p+1} - x^{p}  + (1-\frac1p)\,(x-1).$$
that is, we must show that
$$f_p(x):=  x^p \ge   x  +  (1-\frac1p)(x-1)=:g_p(x) .$$
Now
$$ f_p'(x) = px^{p-1}, \quad g_p'(x) =  2-\frac1p.$$
It follows that $f_p'(x) \ge g_p'(x)\ge 0$ for $x\ge 1$, since $p \ge 2-\frac1p$
(multiply both sides by $p$). Together with $f_p(1) =g_p(1)$  it follows that $f_p(x)\ge g_p(x)$ for $x\ge 1$ and $p\ge 1$.
This concludes the proof of (iii).
 
\smallskip 
(iv): \ Fix $m\ge 1$. For $x\in [-\frac12,\frac12]$ we have
\begin{align*}
 & (h(x+\tfrac12+m)+h(x+\tfrac12-(m+1)))^r - h(x+\tfrac12+m)^r-h(x+\tfrac12-(m+1))^r  
 \\ &  = \left[\Big(\frac{1}{(x+(m+\frac12))^{2}} + \frac{1}{(x-(m+\frac12))^{2}}\Big)^r  
 -\Big(\frac{1}{(x+(m+\frac12))^{2}}\Big)^r - \Big(\frac{1}{(x-(m+\frac12))^{2}}\Big)^r\right]
 \\ & \hskip11.8cm\times \pi^{-2r}\big(\cos^{2}(\pi x)\big)^r.
\end{align*}
We must show that this function has a global maximum on $[-\frac12,\frac12]$ at $x=0$. Since by Lemma~\ref{lem:6} ${\cos(\pi x)}/(1-4x^2)$ has a global maximum at $x=0$, it suffices to prove that 
\begin{equation*}
\left[((a+x)^{-2}+(a-x)^{-2})^r-(a+x)^{-2r}-(a-x)^{-2r}\right]\cdot (1-4x^2)^{2r}
\end{equation*}
has a global maximum at $x=0$, where we have written $a:=m+\frac{1}{2}\geq \tfrac{3}{2}$. Since the 
function $x\mapsto x^r$ is convex, we have $\tfrac{1}{2}(a+x)^{-2r}+\tfrac{1}{2}(a-x)^{-2r}\geq 
(\tfrac{1}{2}(a+x)^{-2}+\tfrac{1}{2}(a-x)^{-2})^r$ and hence 
$$2^{1-r}((a+x)^{-2}+(a-x)^{-2})^r-(a+x)^{-2r}-(a-x)^{-2r}\leq 0$$
with equality for $x=0$. Therefore, it suffices to show that
\begin{equation*}
(1-2^{1-r})((a+x)^{-2}+(a-x)^{-2})^r(1-4x^2)^{2r}
\end{equation*}
has a global maximum at $x=0$. It is enough to show that $g(x):=((a+x)^{-2}+(a-x)^{-2})(1-4x^2)^2$ 
is decreasing on $[0,\tfrac{1}{2}]$.

Computing the derivative of $g$ we find
\begin{align*}
g'(x)&=-16x(1-4x^2)((a+x)^{-2}+(a-x)^{-2})+(1-4x^2)^2(-2(a+x)^{-3}+2(a-x)^{-3})\\
&=(1-4x^2)(a+x)^{-3}(a-x)^{-3}k(x),
\end{align*}
where 
\begin{align*}
k(x)&=-16x(a^2-x^2)((a+x)^2+(a-x)^2)+(1-4x^2)(2(a+x)^3-2(a-x)^3)\\
&=-16x\cdot2(a^4-x^4)+(1-4x^2)\cdot 4x\cdot (3a^2+x^2)\\
&=4x[-8(a^4-x^4)+(1-4x^2)(3a^2+x^2)]\\
&=4x[4x^4+(1-12a^2)x^2+(3a^2-8a^4)]. 
\end{align*}

Since $a>\sqrt{\tfrac{3}{8}}$, the function $p(y):=4y^2+(1-12a^2)y+(3a^2-8a^4)$ has a positive and a 
negative root. The sum of the two roots equals $\tfrac{12a^2-1}{4}$ and therefore the positive root 
is larger than $3a^2-\tfrac{1}{4}\geq \tfrac{26}{4}$. It follows that $p$ is negative on $[0,\tfrac{1}{4}]$ 
and hence $g'(x)=(1-4x^2)(a+x)^{-3}(a-x)^{-3}\cdot4x\cdot p(x^2)\leq 0$ on $[0,\tfrac{1}{2}]$, which finishes the proof.
\end{proof}

\bigskip\noindent
{\bf Added in proof}. 
After this paper had been accepted for publication, Tom Koornwinder sent us the following 
interesting proof for the case that the parameter $r$ in Proposition~\ref{prop} is integral. 
With his kind permission we reproduce it here.

We consider $f_r(x)$ on $(0,1)$. In terms of the Hurwitz zeta-function $\zeta(s,q)$ (see \cite[Eq. 25.11.1]{nist}) we have 
\begin{equation*}
f_r(x)=\pi^{-2r} \sin^{2r}(\pi x) (\zeta(2r,x)+\zeta(2r,1-x)),\quad r=1,2,\ldots
\end{equation*}
In terms of the digamma function $\psi(z) = \Gamma'(z)/\Gamma(z)$ (see \cite[Eq. 25.11.12]{nist}) this 
can be rewritten as 
\begin{equation*}
f_r(x)=\frac{\pi^{-2r} \sin^{2r}(\pi x)}{(2r-1)!} \ddx^{\!2r-1}(\psi(x)-\psi(1-x)).
\end{equation*}
Applying the reflection formula $\psi(1-z)-\psi(z)=\pi\cot(\pi z)$ (see \cite[Eq. 5.5.4]{nist}) we obtain

\begin{equation*}
f_r(x)=\frac{-\pi^{1-2r} \sin^{2r}(\pi x)}{(2r-1)!} \ddx^{\!2r-1}\cot(\pi x).
\end{equation*}
Substitution of $t=\pi x$ simplifies this expression to 
\begin{equation*}
\frac{(2r-1)!}{\sin^{2r}t}\cdot f_r(t/\pi)  = - \ddt^{\!2r-1} \cot t.
\end{equation*}

Since $(d/dt)\cot t=-1/\sin^2 t$, we have $f_1(t/\pi)=1$. Also, we obtain the following recursion relation:
\begin{equation}
\frac{(2r+1)!}{\sin^{2r+2} t }\cdot f_{r+1}(t/\pi) =  -  \ddt^{\!2r+1} \cot t
\       =\ (2r-1)!\,  \ddt^{\!2} \frac{f_r(t/\pi)}{\sin^{2r}t}.\label{recursion2}
\end{equation}
A small computation shows that 
\begin{align*}
\ddt^{\!2} \frac{f_r(t/\pi)}{\sin^{2r}t}&=\ddt\left[(\sin^{-2r} t)\ddt f_r(t/\pi)-2r(\cos t)(\sin^{-2r-1} t)f_r(t/\pi)\right]&\\
&=(\sin^{-2r} t)\ddt^{\!2} f_r(t/\pi)-4r(\cos t)(\sin^{-2r-1} t)\ddt f_r(t/\pi)\\
&\qquad\qquad\, +\left(2r(2r+1)\cos^2 t\sin^{-2r-2} t+2r\sin^{-2r} t\right)\cdot f_r(t/\pi).
\end{align*}
Hence, (\ref{recursion2}) implies that 
\begin{align}\label{recursion3}
\frac{(2r+1)!}{(2r-1)!}f_{r+1}(t/\pi)=(\sin^2 t)\ddt^{\!2} f_r(t/\pi)-4r(\cos t)(\sin t)\ddt f_r(t/\pi)&\nonumber\\
+(2r(2r+1)\cos^2 t+2r\sin^2 t)\cdot f_r(t/\pi)&.
\end{align}
Set $y:=\cos^2 t$ and $D:=d/dy$. So $d/dt=-2(\sin t\cos t)D$ and 
\begin{align*}\ddt^{\!2}&=\frac{d}{dt}(-2\sin t\cos t)D\\
&=-2(\cos^2 t-\sin^2 t)D -(2\sin t\cos t)\ddt D\\
&=-2(\cos^2 t-\sin^2 t)D +(4\sin^2 t\cos^2 t) D^2\\
&=(-4y+2)D+4y(1-y)D^2.
\end{align*}
Equation (\ref{recursion3}) can therefore be rewritten as 
\begin{align}\frac{(2r+1)!}{(2r-1)!}\, f_{r+1}(t/\pi)
      &  =  \left[4y(1-y)^2 D^2+((8r-4)y+2)(1-y) D + 2r(2ry+1)\right] f_r(t/\pi)\nonumber\\
 &  = \left[4y(r-yD)^2 + 8(r-yD)yD+2yD+2r + 4yD^2+2D\right] f_r(t/\pi).\label{recursion4}
\end{align}
Observe that $(r-yD)y^k=(r-k)y^k$. Hence, if $p=p(y)$ is a polynomial of degree $n<r$ with nonnegative coefficients, then the same holds for $(r-yD)p$. The recursion (\ref{recursion4}) and the fact that $f_1(t/\pi)=1$ now imply that $f_r(t/\pi)$ is a polynomial in $y$ of degree $r-1$ with nonnegative coefficients. 
The first few are given explicitly by
\begin{align*}
f_1(t/\pi) & = 1  \\ 
f_2(t/\pi) & = \frac13 + \frac23 \cos^2 t \\
f_3(t/\pi) & = \frac2{15} + \frac{11}{15} \cos^2 t + \frac2{15} \cos^4 t \\
f_4(t/\pi) & = \frac{17}{315} + \frac{4}{7} \cos^2 t + \frac{38}{105} \cos^4 t + \frac{4}{315} \cos^6 t \\
f_5(t/\pi) & = \frac{62}{2835} + \frac{1072}{2835} \cos^2 t + \frac{484}{945} \cos^4 t + \frac{247}{2835} \cos^6 t  + \frac{2}{2835} \cos^8 t    
\end{align*}
For integers $r$, Proposition~\ref{prop} is an immediate consequence.

For half-integers $r = n+\frac12$ one could observe that the identity
$$ \psi^{(2n)}(x) = - (2n)! \sum_{m=0}^\infty \frac1{(x+m)^{2n+1}}$$
allows one to express the inequality of Proposition~\ref{prop} in terms of
the polygamma functions $ \psi^{(2n)}$. We have not been able, however, to use this fact to give a simpler proof in that case.


\begin{thebibliography}{9}

\bibitem{Bou}{J. Bourgain, 
Vector-valued Hausdorff-Young inequalities and applications. 
In: Geometric aspects of functional analysis (1986/87), 239--249,
Lecture Notes in Math., Vol. 1317, Springer, Berlin, 1988.}
 
\bibitem{Gar}{J. Garcia-Cuerva, K.S. Kazarian, V.I. Kolyada, \& J.L. Torrea, 
Vector-valued Hausdorff-Young inequality and applications. 
Russian Mathematical Surveys, {\bf 53}(3) (1998), 435--513.} 

\bibitem{Kon}{H. K\"onig, 
On the Fourier-coefficients of vector-valued functions.
Math. Nachr. {\bf 152} (1991), 215--227.}
 
\bibitem{Kwa}{S. Kwapie\'n, 
Isomorphic characterizations of inner product spaces by orthogonal series with vector valued coefficients. 
Studia Math. {\bf 44} (1972), 583--595.}

\bibitem{Pee}{J. Peetre, 
Sur la transformation de Fourier des fonctions \`a valeurs vectorielles.
Rend. Sem. Mat. Univ. Padova {\bf 42} (1969), 15--26.} 

\bibitem{nist}{{\tt http://dlmf.nist.gov}}
\end{thebibliography}
\end{document}